\documentclass[12pt,leqno]{amsart}
\usepackage{latexsym,amssymb,amsmath,multicol, rotating,lscape}
\usepackage{letltxmacro}

\usepackage{tikz}
\usetikzlibrary{calc,decorations.markings}

 2

\newtheorem{thm}{Theorem}[section]
\newtheorem{lem}[thm]{Lemma}

\newtheorem{cor}[thm]{Corollary}
\newcommand{\thmref}[1]{Theorem~\ref{#1}}
\newcommand{\lemref}[1]{Lemma~\ref{#1}}

\theoremstyle{remark}
\newtheorem{rmk}{Remark}[section]

\begin{document}

\title[$\ell$-fold product $L$-function]
{On the coefficients of $\ell$-fold product $L$-function}

\author[Ayyadurai Sankaranarayanan]{Ayyadurai Sankaranarayanan \textsuperscript{(1)}}
\author[Lalit Vaishya]{Lalit Vaishya \textsuperscript{(2)}}

\address[1]{School of Mathematics and Statistics, University of Hyderabad, CR Rao Road, Hyderabad, Telangana
500046, India}
\email{sank@uohyd.ac.in}

\address[2]{The Institute of Mathematical Sciences (A CI of Homi Bhabha
National Institute)
CIT Campus, Taramani,
Chennai - 600 113
Tamil Nadu 
India}
\email{lalitvaishya@gmail.com, lalitv@imsc.res.in}


\date{\today}
 
\maketitle

\begin{abstract}
Let	 $f \in S_{k}(SL_2(\mathbb{Z}))$ be a normalized Hecke eigenforms of integral weight $k$ for the full modular group. In the article, we study the average behaviour of Fourier coefficients of $\ell$-fold product $L$-function. More precisely, we establish the asymptotics of power moments associated to the sequence  $\{\lambda_{f \otimes f \otimes \cdots \otimes_{\ell} f}(n)\}_{n- {\rm squarefree}}$ where ${f \otimes f \otimes \cdots \otimes_{\ell} f}$ denotes the $\ell$-fold product of $f$.  As a consequence, we prove results concerning the behaviour of sign changes  associated to these sequences for odd $\ell$-fold product $L$-function. A similar result also holds for the sequence $\{\lambda_{f \otimes f \otimes \cdots \otimes_{\ell} f}(n)\}_{n \in \mathbb{N}}$.  
\end{abstract}

\section{Introduction}

Let $S_{k}(SL_2(\mathbb{Z}))$ denote the ${\mathbb C}$-vector space of cusp forms of weight $k$ for the full modular group $SL_2(\mathbb{Z})$. A cusp form $f \in S_{k}(SL_2(\mathbb{Z}))$ is said to be a Hecke eigenform if $f$ is a simultaneous eigenfunction for all the Hecke operators. Let  $a_{f}(n)$ denote the  $n^{\rm th}$ Fourier coefficient of a cusp form $f \in S_{k}(SL_2(\mathbb{Z}))$. A cusp form $f$ is said to be normalised if $a_f(1)=1$. We define the normalised $n^{\rm th}$ Fourier coefficients $ \lambda_f(n)$ given by; $ \lambda_f(n) := {a_{f}(n)}/{n^{\frac{k-1}{2}}}$.  The normalised Fourier coefficient $\lambda_{f}(n)$ is a multiplicative function and satisfies the following recursive relation \cite[Eq. (6.83)]{Iwaniec}:
\begin{equation}
\lambda_{f}(m)\lambda_{f}(n) = \sum_{d \vert m,n }  \lambda_{f}\left(\frac{mn}{d^2}\right),  
\end{equation}
 for all positive integers $m$ and $n.$ Ramanujan conjecture predicts that $|\lambda_{f}(p)| \le 2$.  It has been established in a pioneer work of  Deligne. More precisely, it has been proved that 
\begin{equation} \label{lambda-coefficient-bound}
|\lambda_{f}(n)| \le d(n) \ll_{\epsilon} n^{\epsilon}, 
\end{equation}
for  any arbitrary small  $\epsilon > 0,$ where $d(n)$ denotes the number of positive divisors of $n$. 

\smallskip
Let  $f(\tau) = \displaystyle{\sum_{n=1}^\infty \lambda_f(n)n^{\frac{k-1}{2}}q^{n}} \in S_{k}(SL_2(\mathbb{Z}))$ be a normalised Hecke eigenform. We define the Hecke $L$-function   given by (for $ \Re(s) >1$)
\begin{equation}
\begin{split}
L(s, f) &= \sum_{n \ge 1} \frac{\lambda_{f}(n)}{n^{s}} =  \prod_{p} \left(1-\frac{\lambda_f(p)}{p^s}-\frac{1}{p^{2s}}\right)^{-1}  = \prod_{p} \left(1-\frac{\alpha_{f}(p)}{p^s}\right)^{-1}\left(1-\frac{\beta_{f}(p)}{p^{s}}\right)^{-1} \!\!\!\!\!,
\end{split}
\end{equation}
where,  for any prime $p$, there exist complex numbers $\alpha_{f}(p)$ and $\beta_{f}(p)$ such that   
 \begin{equation}\label{PropertyCoefficients}
\begin{split}
\alpha_{f}(p)+\beta_{f}(p)=\lambda_f(p) \quad {\rm  and} \quad   |\alpha_{f}(p)|= |\beta_{f}(p)| = \alpha_{f}(p)\beta_{f}(p)=1.
\end{split}
\end{equation}
The Hecke $L$-function $L(s, f)$ satisfies a nice functional equation and it has analytic continuation to whole $\mathbb{C}$-plane \cite[Section 7.2]{Iwaniec}.

Following the work of  Garret and Harris \cite{Garret-Harris}, associated to Hecke eigenforms $f_{1}, f_{2}, \cdots, f_{\ell}$ of weight $k$ for the full modular group $SL_2(\mathbb{Z})$, we define the $\ell$-fold product  $L$-function $L(s, f_{1} \otimes f_{2} \otimes \cdots \otimes f_{\ell})$ given by (for $\Re(s)>1$)
\begin{equation}
\begin{split}
L(s, f_{1} \otimes f_{2} \otimes \cdots \otimes f_{\ell}) &:=  \sum_{n \ge 1} \frac{\lambda_{f_{1} \otimes f_{2} \otimes \cdots \otimes f_{\ell}}(n)}{n^{s}} \\
&=  \prod_{p-{\rm prime}} \prod_{\sigma}\left(1-{\alpha_{f_{1}}^{\sigma(1)}(p)} {\alpha_{f_{2}}^{\sigma(2)}(p)} \cdots {\alpha_{f_{\ell}}^{\sigma(\ell)}(p)} {p^{-s}}\right)^{-1} 
\end{split}
\end{equation}
 where $\sigma$ runs over the set of maps from  $\{1, 2, \cdots \ell\}$ to $\{1, 2\}$ and ${\alpha_{f_{j}}^{\sigma(i)}(p)} = {\alpha_{f_{j}}(p)}$ if $\sigma(i) =1$ and ${\alpha_{f_{j}}^{\sigma(i)}(p)} = {\beta_{f_{j}}(p)}$ if $\sigma(i) =2$ satisfying \eqref{PropertyCoefficients} corresponding to $f_{j}$.

For each $j$, let $f_{j} =f \in S_{k}(SL_2(\mathbb{Z}))$. Then, we consider the following $\ell$-fold product $L$-function associated to $f$ given by   
\begin{equation}\label{R-L-functionDef}
\begin{split}
L(s, f \otimes f \otimes \cdots \otimes_{\ell} f) &:=  \sum_{n \ge 1} \frac{\lambda_{f \otimes f \otimes \cdots \otimes_{\ell} f}(n)}{n^{s}} \\
&=  \prod_{p-{\rm prime}} \prod_{\sigma}\left(1-{\alpha_{f}^{\sigma(1)}(p)} {\alpha_{f}^{\sigma(2)}(p)} \cdots {\alpha_{f}^{\sigma(\ell)}(p)} {p^{-s}}\right)^{-1} 
\end{split}
\end{equation}
where $\sigma$ runs over the set of maps from  $\{1, 2, \cdots \ell\}$ to $\{1, 2\}$.
At a prime $p$, it is easy to observe that  the Fourier coefficient of $\ell$-fold product $L$-function is the $\ell^{\rm th}$-power of $\lambda_{f}(p)$, i.e.,
\begin{equation}\label{SymLf}
\begin{split}
\lambda_{f \otimes f \otimes \cdots \otimes_{\ell} f}(p) = \lambda_{f}^{\ell}(p).
\end{split}
\end{equation}

A classical problem in analytic number theory is to study the average behaviour and  distribution of arithmetical functions. One of the interesting object in consideration is the arithmetical functions associated to Fourier coefficients of automorphic forms.  For a given sequence $\{x_{n}\}$ in $\mathbb{N}$ and an arithmetical function $n \mapsto A(n)$,  one of the intriguing problem is to study the power moments associated to the sequence $\{A(x_{n})\}_{n \in \mathbb{N}}$. In this regard, several interesting results for $A(n)= \lambda_{f}(n)$  and $\lambda_{\pi}(n)$, where $f$ is a $GL(2)$-form and $\pi$ is an automorphic representation,  have been established by several mathematicians. Recently, in a joint work with  Venkatasubbareddy, the first author \cite{KVAS} established an estimate for the following  arithmetical functions $n \mapsto \lambda_{f \otimes f\otimes f}(n)$ and  $n \mapsto \lambda_{ f\otimes {\rm sym^{2}}f}(n)$, associated to   a normalized Hecke eigenform $f \in S_{k}(SL_2(\mathbb{Z}))$ and improved the bounds of his previous work with  L\"{u}  \cite{GSLAS}. 

\smallskip
In  \cite{Lalit-M},  the second author considered the same  arithmetical functions $n \mapsto \lambda_{f \otimes f\otimes f}(n)$ and  $n \mapsto \lambda_{ f\otimes {\rm sym^{2}}f}(n)$, associated to   a normalized Hecke eigenform $f$ of weight $k$ for the congruence subgroup  $\Gamma_{0}(N)$,  and investigate the oscillations  of the sequences $\{\lambda_{f \otimes f\otimes f}(n)\}$ and  $\{\lambda_{f \otimes {\rm sym^{2}}f}(n)\}$ where the indices are supported on the square-free integers represented by the primitive integral positive-definite binary quadratic forms (reduced forms) of discriminant $D$.  More precisely, the second author established certain estimate  the following sums (for $\ell=1,2$): 
\begin{equation*}
\begin{split}
 \sideset{}{^{\flat }}\sum_{\substack{\mathcal{Q}(\underline{x}) \le X \\  ~\mathcal{Q} \in  \mathcal{S}_{D}, \underline{x} \in \mathbb{Z}^{2} \\ } } \!\!\!\!  \!\!\!\! \lambda_{f \otimes f \otimes f}^{\ell}(\mathcal{Q}(\underline{x}))
\qquad {\rm and} \qquad
\sideset{}{^{\flat }}\sum_{\substack{\mathcal{Q}(\underline{x}) \le X \\  ~\mathcal{Q} \in  \mathcal{S}_{D}, \underline{x} \in \mathbb{Z}^{2} \\} }  \!\!\!\! \!\!\!\!  \lambda_{ f \otimes {\rm sym^{2}}f}^{\ell}(\mathcal{Q}(\underline{x})). \\
\end{split}
\end{equation*}
where $\sideset{}{^{\flat }}\sum$ means that the sum is supported on square-free positive integers and $\mathcal{S}_{D}$ denotes the set of inequivalent reduced forms of fixed discriminant $D$. As  a consequence, the author proved the behaviour of sign changes of the above sequences.

In \cite{KVASankar}, the first author studied the average behaviour of Fourier coefficients associated to  tetra, penta, hexa, hepta and octa product $L$-functions and improve the previous results.

In this article, for a given Hecke eigenform $f \in S_{k}(SL_{2}(\mathbb{Z}))$,  we consider the arithmetical functions $n \mapsto \lambda_{f \otimes f \otimes \cdots \otimes_{\ell} f}(n)$ where ${f \otimes f \otimes  \cdots \otimes_{\ell} f}$ denotes the $\ell$-fold product of $f$. We  study the oscillations of the sequence $\{\lambda_{f \otimes f \otimes \cdots \otimes_{\ell} f}(n)\}$ where the sequence is supported on the set of square-free integers. More precisely, we establish certain estimate for the following sums.
\begin{equation}\label{Firstsum}
\begin{split}
S_{\ell}(f, X ) &=  \sideset{}{^{\flat }}\sum_{\substack{n \le X}}  \lambda_{f \otimes f \otimes  \cdots \otimes_{\ell} f} (n) 
\quad {\rm and } \quad
T_{\ell}(f, X ) =  \sideset{}{^{\flat }}\sum_{\substack{n \le X}}  \lambda_{f \otimes f \otimes  \cdots \otimes_{\ell} f}^{2} (n)
\end{split}
\end{equation}
where $\sideset{}{^{\flat }}\sum$ means that the sum is supported on square-free positive integers.

Throughout the paper, $\epsilon$ occurring at various places below  is not same everywhere  but a positive function of $\epsilon>0$, and for $\ell \in \mathbb{N}$ and $r \le \ell$,  ${\ell \choose r }$ denotes the binomial coefficient with the convention  ${\ell \choose r } =0$ if $r <0$ , and  $[x]$ denotes the greatest integer $ \le x$.

\bigskip
Now, we state our results.
\begin{thm}\label{Approx1}
Let  $\ell \ge 3$ be a positive integer.
Let  $f \in S_{k}(SL_2(\mathbb{Z}))$ be a normalised Hecke eigenform. Then, for any $\epsilon>0,$ we have the following estimates for the sums $S_{\ell}(f, X )$ defined in \eqref{Firstsum}. For odd $\ell$, 
\begin{equation}\label{EstA}
\begin{split}
S_{\ell}(f, X ) =  \sideset{}{^{\flat }}\sum_{\substack{n \le X}}  \lambda_{f \otimes f \otimes  \cdots \otimes_{\ell} f} (n) =  O\left( X^{1- \frac{1}{\alpha_{\ell}}+ \epsilon} \right)
\end{split}
\end{equation}
where  $\alpha_{\ell} = \frac{2}{3([\frac{\ell}{2}]+2)} {\ell \choose {[\frac{\ell}{2}]}} + \frac{1}{2} \displaystyle{ \left[\sum_{n=0}^{[\ell/2]-1} {\frac{(\ell-2n+1)^{2}}{\ell-n+1} {\ell \choose {n}}} \right]} .$
For even $\ell$, 
\begin{equation}\label{EstS}
\begin{split}
S_{\ell}(f, X ) & =  \sideset{}{^{\flat }}\sum_{\substack{n \le X}}  \lambda_{f \otimes f \otimes  \cdots \otimes_{\ell} f} (n) =  X P_{\ell}(\log X) + O\left( X^{1- \frac{1}{\beta_{\ell}}+ \epsilon} \right)
\end{split}
\end{equation}
where $P_{\ell}(X)$ is  a polynomial of degree $ \frac{2}{ (\ell+2)} {\ell \choose {\frac{\ell}{2}}}-1$ with positive coefficients and 
 $$\beta_{\ell} =  \frac{1}{4}+\frac{13}{21 (\ell+2)} {\ell \choose {\frac{\ell}{2}}} +  \frac{15}{2(\ell+4)} {\ell \choose {\frac{\ell}{2}}-1} + \frac{1}{2} \displaystyle{ \left[\sum_{n=0}^{\frac{\ell}{2}-2} {\frac{(\ell-2n+1)^{2}}{\ell-n+1} {\ell \choose {n}}} \right]}.$$
\end{thm}

\begin{rmk}
From  the multiplicativity of $\lambda_{f \otimes f \otimes  \cdots \otimes_{\ell} f} (n)$, we have the following  estimate for $ A_{\ell}(f,  X):= \displaystyle{\sum_{\substack{n \le X}}}  \lambda_{f \otimes f \otimes  \cdots \otimes_{\ell} f} (n)$.
\begin{equation*}
\begin{split}
 A_{\ell}(f,  X) & = \!\!\!\! \sum_{\substack{n = QR \le X  \\ Q- \text{squarefull} \\ R- \text{squarefree} \\  \gcd(Q,R)=1}}  \lambda_{f \otimes f \otimes  \cdots \otimes_{\ell} f} (QR) = \!\!\!\! \sum_{\substack{Q \le X \\ Q- \text{squarefull}}}  \!\!\!\! \!\!\!\! \lambda_{f \otimes f \otimes  \cdots \otimes_{\ell} f} (Q) \!\!\!\! \sum_{\substack{R \le \frac{X}{Q} \\ R- \text{squarefree}}} \!\!\!\! \lambda_{f \otimes f \otimes  \cdots \otimes_{\ell} f} (R).
\end{split}
\end{equation*}
Let $M_{\ell}(X) =0$ or  $X P_{\ell}(X)$, and $\xi_{\ell} = 1-  \frac{1}{\alpha_{\ell}}+\epsilon$ or $ 1-\frac{1}{\beta_{\ell}}+\epsilon$  according as $\ell$ is odd or even for any $\epsilon>0$. It is easy to see that $\xi_{\ell} > \frac{1}{2}$ (for $\ell \ge 3$). Then, from  \thmref{Approx1} and $\lambda_{f \otimes f \otimes  \cdots \otimes_{\ell} f} (n) \ll n^{\epsilon}$, we get 
\begin{equation*}
\begin{split}
 A_{\ell}(f,  X) & = \!\!\!\! \sum_{\substack{Q \le X \\ Q- \text{squarefull}}}  \!\!\!\! \!\!\!\! \lambda_{f \otimes f \otimes  \cdots \otimes_{\ell} f} (Q) \!\!\!\! \sum_{\substack{R \le \frac{X}{Q} \\ R- \text{squarefree}}} \!\!\!\! \lambda_{f \otimes f \otimes  \cdots \otimes_{\ell} f} (R)  \\
 & = \sum_{\substack{Q \le X \\ Q- \text{squarefull}}}  \!\!\!\! \!\!\!\! \lambda_{f \otimes f \otimes  \cdots \otimes_{\ell} f} (Q) \left(M_{\ell}(X/Q) + \left(\frac{X}{Q}\right)^{\xi_{\ell}} \right)\\
 & = \sum_{\substack{Q \le X \\ Q- \text{squarefull}}}  \!\!\!\! \!\!\!\! Q^{\epsilon}  \left(M_{\ell}(X/Q) + \left(\frac{X}{Q}\right)^{\xi_{\ell}} \right).
\end{split}
\end{equation*}
We observe that $Q= q_{1}^{r_{1}} \cdots q_{j}^{r_{j}}$ with each $r_{t} \ge 2$. Let $q = q_{1} \cdots q_{j}.$ Then 
\begin{equation*}
\begin{split}
  \sum_{\substack{Q \le X \\ Q- \text{squarefull}}}  \!\!\!\!  Q^{- \xi_{\ell}+\epsilon}
 \ll \sum_{\substack{q^{2} \le X \\  \text{q-squarefree}}}   q^{- 2\xi_{\ell}+\epsilon} < \infty \quad \text{and} \quad  \sum_{\substack{Q \le X \\ Q- \text{squarefull}}}  \!\!\!\!  Q^{- 1+\epsilon} < \infty.
\end{split}
\end{equation*}
as $\xi_{\ell} > \frac{1}{2}$. Hence, we get 

$$ \sum_{n \le X}  \lambda_{f \otimes f \otimes  \cdots \otimes_{\ell} f} (n) \ll X^{1-\frac{1}{\alpha_{\ell}}+\epsilon} \quad \text{when $\ell$ is odd} $$
and
$$ \sum_{n \le X}  \lambda_{f \otimes f \otimes  \cdots \otimes_{\ell} f} (n) - \text{Main term}\ll X^{1-\frac{1}{\beta_{\ell}}+\epsilon} \quad \text{when $ \ell $ is even}$$  
with different absolute constants. Thus, we obtain
\begin{equation}\label{ErrT1}
\begin{split}
\sum_{n \le X}  \lambda_{f \otimes f \otimes  \cdots \otimes_{\ell} f} (n)  =
\begin{cases}
O\left( X^{\frac{7}{10}+ \epsilon} \right) \quad {\rm if } \quad \ell = 3, \\
O\left( X^{\frac{33}{38}+ \epsilon} \right) \quad {\rm if } \quad \ell = 5, \\
O\left( X^{\frac{161}{164}+ \epsilon} \right) \quad {\rm if } \quad \ell = 7,\\
\end{cases}
\end{split}
\end{equation}
and
\begin{equation}\label{ErrT2}
\begin{split}
\sum_{n \le X}  \lambda_{f \otimes f \otimes  \cdots \otimes_{\ell} f} (n)  = X Q_{\ell}(\log X) + 
\begin{cases}
O\left( X^{\frac{257}{299}+ \epsilon} \right) \quad {\rm if } \quad \ell = 4, \\
O\left( X^{\frac{589}{610}+ \epsilon} \right) \quad {\rm if } \quad \ell = 6, \\
O\left( X^{\frac{1411}{1423}+ \epsilon} \right) \quad {\rm if } \quad \ell = 8. \\
\end{cases}
\end{split}
\end{equation}
where $Q_{\ell}(t)$ is a polynomial of degree $1, 4$ and $13$ for $\ell =4, 6$ and $8$, respectively. It  allows us  to improve an earlier result in \cite{KVASankar}. Let 
\begin{equation*}
\begin{split}
\sum_{n \le X}  \lambda_{f \otimes f \otimes  \cdots \otimes_{\ell} f} (n)  = M_{\ell}(X) + E_{\ell}(X). 
\end{split}
\end{equation*}
Then,  comparing the earlier results of \cite{KVASankar},  we notice that  we get a better  error term, i.e.,
\begin{equation*}
\begin{split}
E_{5}(X) &\ll  X^{\frac{33}{38}+ \epsilon} \ll  X^{\frac{40}{43}+ \epsilon}, \quad E_{7}(X) \ll X^{\frac{161}{164}+ \epsilon} \ll  X^{\frac{184}{187}+ \epsilon} \quad  {\rm and},   \\
E_{4}(X) &\ll X^{\frac{257}{299}+ \epsilon} \ll   X^{\frac{79}{91}+ \epsilon}, \quad E_{6}(X) \ll  X^{\frac{589}{610}+ \epsilon} \ll  X^{\frac{88}{91}+ \epsilon} \quad  {\rm and} \quad \\
 E_{8}(X) &\ll  X^{\frac{1411}{1423}+ \epsilon} \ll  X^{\frac{374}{377}+ \epsilon}. \\
\end{split}
\end{equation*}
Thus, \eqref{ErrT1} and \eqref{ErrT2} improve Theorem 2.1 and Theorem 2.2 of \cite{KVASankar}. However, we note that $E_3(X) \ll X^{\frac{7}{10}+\epsilon}$ is the same bound obtained by L\"{u} and Sankaranarayanan in \cite{GSLAS}.
\end{rmk}

\begin{thm}\label{Approx2}
Let  $\ell \ge 3$ be a positive integer.
Let  $f \in S_{k}(SL_2(\mathbb{Z}))$ be a normalised Hecke eigenform. Then, for any $\epsilon>0,$ we have
\begin{equation}\label{EstT}
\begin{split}
T_{\ell}(f, X )  &=  \sideset{}{^{\flat }}\sum_{\substack{n \le X}}  \lambda_{f \otimes f \otimes  \cdots \otimes_{\ell} f}^{2} (n) =  X P_{2\ell}(\log X) + O\left( X^{1- \frac{1}{\beta_{2\ell}}+ \epsilon} \right)
\end{split}
\end{equation}
where $P_{\ell}(X)$ is  a polynomial of degree $ \frac{2}{ (\ell+2)} {\ell \choose {\frac{\ell}{2}}}-1$ with positive coefficients and  $\beta_{\ell}$ is as given in \thmref{Approx1}.

\end{thm}


As a consequence,  for an odd $\ell$,  we investigate  the  behaviour of the sign changes of the sequences $\{\lambda_{f \otimes f \otimes \cdots \otimes_{\ell} f}(n)\}_{n- {\rm squarefree}}$ and  establish the result on the number of sign changes in the short interval. Moreover, as an application, we prove that  there are infinitely many  sign changes of the above mentioned sequences.

\begin{thm}\label{N sign change}
Let  $f \in S_{k}(SL_2(\mathbb{Z}))$ be a normalised Hecke eigenform and $\ell$ be an odd positive integer. Let   $X$ be sufficiently large  real number  and $\epsilon>0$ be arbitrarily small real number. Let  $h= X^{1- \delta}$ with $  \frac{1}{\beta_{2\ell}} \le \delta < \frac{1}{\alpha_{\ell}}$. Then, the sequence $\{\lambda_{f \otimes f \otimes \cdots \otimes_{\ell} f}(n)\}_{n- {\rm squarefree}}$ has always a sign change at some square-free integer in the interval $[X, X+h]$.  Moreover, there are (at least) $X^{\delta}$ many sign changes with$  \frac{1}{\beta_{2\ell}} \le \delta < \frac{1}{\alpha_{\ell}}$, in the interval $[X,2X]$. 
\end{thm}	

\begin{cor} Let  $\ell$ be an odd positive integer. Then,
 the sequence $\{\lambda_{f \otimes f \otimes \cdots \otimes_{\ell} f}(n)\}_{n- {\rm squarefree}}$ has infinitely many sign changes. 
\end{cor}

\begin{rmk}
 The above result can also be proved for any congruence subgroup $\Gamma_{0}(N)$ of $SL_2(\mathbb{Z})$ with an error term in terms of analytic conductor also. For simplicity, we prove our result for the full modular group  $SL_2(\mathbb{Z})$. 
\end{rmk}

\smallskip
The paper is organized as follows. In the next section, we introduce key ingredients  which leads to establish our results. Finally, in last section, we prove our results.

\section{Key Ingredients}
In order to handle these sums $S_{\ell}(f, X )$ and $T_{\ell}(f, X )$ defined in \eqref{Firstsum}, we define  the Dirichlet series
\begin{equation}\label{R-L-function1}
\begin{split}
L_{S}(s) &:=  \sideset{}{^{\flat }} \sum_{n \ge 1 } \frac{\lambda_{f \otimes f \otimes \cdots \otimes_{\ell} f}(n)}{n^{s}}
\qquad {\rm and } \qquad 
  L_{T}(s) :=  \sideset{}{^{\flat }} \sum_{n \ge 1} \frac{\lambda_{f \otimes f \otimes \cdots \otimes_{\ell} f}^{2}(n)}{n^{s}}.\\
\end{split}
\end{equation}
The Dirichlet series for $L_{S}(s)$ and $L_{T}(s)$  converge absolutely and uniformly for $\Re(s) >1.$   To obtain an asymptotic formula with a good error term for the sums $S_{\ell}(f, X )$ and $T_{\ell
}(f, X )$ defined in \eqref{Firstsum}, we first decompose $L_{S}(s)$ and $L_{T}(s)$ in terms of known $L$-functions. Using the analytic properties of well-known $L$-functions, we achieve our goal.

For each $m \ge 2$, we define  the $m^{th}$ symmetric power $L$-function associated to a normalized Hecke eigenform $f \in S_{k}(SL_2(\mathbb{Z}))$ given by 
\begin{equation*}\label{Symf}
\begin{split}
L(s,sym^{m}f)&: = \! \prod_{p} \prod_{j=0}^{m} \left(1- \frac{{\alpha_{f}(p)}^{m-j}{\beta_{f}(p)}^{j}}{p^{-s}}\right)^{-1}  \!\!\!\!
= \zeta(ms)\sum_{n=1}^\infty \frac{\lambda_{f}(n^{m})}{n^s}
=\sum_{n=1}^\infty \frac{\lambda_{sym^{m}f}(n)}{n^s},
\end{split}
\end{equation*}
where $\alpha_{f}(p)$ and $\beta_{f}(p)$ are complex numbers satisfying  \eqref{PropertyCoefficients} and   ${\lambda_{sym^{m}f}(n)}$ is a multiplicative function. For each prime $p$, we have
\begin{equation}\label{SymPowerLf}
\begin{split}
{\lambda_{sym^{m}f}(p)} = \displaystyle{\sum_{j=0}^{m} {\alpha_{f}(p)}^{m-j}{\beta_{f}(p)}^{j}}.
\end{split}
\end{equation}
The Archimedean factors of $L(s,sym^{m}f)$ is defined as 
	\begin{equation}
	\begin{split}
	L_{\infty}(s, sym^{m}f) & = \begin{cases}
	\displaystyle{\prod_{v=0}^{p}} \Gamma_{{\mathbb C}}\left(s +\left(v+\frac{1}{2}\right)(k-1)\right) {~~~~ \rm if ~~~~}  m = 2p +1, \\
	\Gamma_{\mathbb R}(s +\delta_{2 \nmid p}) \displaystyle{\prod_{v=1}^{p}} \Gamma_{{\mathbb C}}(s +v(k-1)) {~~~~ \rm if ~~~~}  m = 2p, \\
	\end{cases}
	\end{split}
	\end{equation}
	where $\Gamma_{\mathbb R}(s) = {\pi}^{-s/2} \Gamma(s/2)$ and $\Gamma_{\mathbb C}(s) = 2 ({2\pi})^{-s} \Gamma(s)$ and, $\delta_{2 \nmid p}=1$ if  $2 \nmid p$ and $0$ if $2 \vert p$.
	We define the completed $L$-function 
	$$\Lambda(s, sym^{m}f) : = L_{\infty}(s, sym^{m}f) L(s, sym^{m}f).$$
We know that 	$\Lambda(s, sym^{m}f)$  is an entire function on whole ${\mathbb C}$-plane and satisfies a nice functional equation
	$$
	\Lambda(s, sym^{m}f)  = \epsilon_{sym^{m}f} \Lambda(1-s, sym^{m}f)  
	$$
	where $\epsilon_{sym^{m}f} = \pm1$. For details, we refer to \cite[Section $3.2.1$]{Cog-Mic}. From Deligne's bound, it is well-known that 
$$
|{\lambda_{sym^{m}f}(n)}| \le d_{m+1}(n) \ll_{\epsilon} n^{\epsilon}
$$
for any real number $\epsilon >0$ and $d_{m}(n)$ denotes the number of $m$ positive factors of $n$. 
\smallskip

\begin{rmk}
	For a classical holomorphic Hecke eigenform $f$, J. Cogdell and P. Michel \cite{Cog-Mic}  have given the explicit description of analytic continuation and functional equation  for the $L$-function $L(s,sym^{m}f)$, $ m \in \mathbb{N}$. Newton and Thorne \cite{{Newton}, {NewThorne}}  established  that $sym^{m}f$ is a cusp form on $GL_{m+1}(\mathbb{A}_{\mathbb{Q}})$, for any positive integer $m$ where $\mathbb{A}_{\mathbb{Q}}$ is the ring of Adeles of $\mathbb{Q}$. So, the explicit description of analytic continuation and functional equation  for the $L$-function $L(s,sym^{m}f)$, $m\in \mathbb{N}$ is well-known.
\end{rmk}
 
\smallskip
Let  $\zeta(s)=  \displaystyle{\sum_{n \ge 1} n^{-s}}$  denote the Riemann zeta function. We assume the following conventions: \quad $L(s, sym^{0}f) = \zeta(s)$ and $L(s, sym^{1}f  ) = L(s, f)$. With these conventions, we state the decomposition of $ L_{S}(s)$ and $ L_{T}(s)$.

\begin{lem}\label{LDecomp}
Let  $ \ell \in \mathbb{N}.$ we have the following decomposition:
 \begin{equation} \label{DecompositionL}
\begin{split}
 L_{S}(s) & = L_{\ell}(s) \times U_{\ell}(s) \\ 
 \end{split}
\end{equation}
where for each odd $\ell$,
\begin{equation}\label{LOdd}
\begin{split}
L_{\ell}(s) & =   \prod_{n=0}^{[\ell/2]} \left({L(s, sym^{\ell-2n}f)}^{\left({\ell \choose n}- {\ell \choose {n-1}}\right)}   \right),       \\
\end{split}
\end{equation}
and for each even  $\ell$,
 \begin{equation}\label{LEven}
\begin{split}
L_{\ell}(s) & = \zeta(s)^{\left({\ell \choose \ell/2}- {\ell \choose {\ell/2-1}}\right)}  \prod_{n=0}^{[\ell/2]-1} \left({L(s, sym^{\ell-2n}f)}^{\left({\ell \choose n}- {\ell \choose {n-1}}\right)}   \right),       \\
\end{split}
\end{equation}
and ${\ell \choose n }$ is the binomial coefficient with the convention ${\ell \choose n } =0$ if $n<0,$ $U_{\ell}(s)$  is given in  terms of an Euler product which converges absolutely and uniformly for $\Re(s)>\frac{1}{2}$ and $U_{\ell}(s) \neq 0$ for $\Re(s)=1.$
\end{lem}
Similarly, we prove the decomposition of $L_{T}(s)$ in the following Lemma.
\begin{lem}\label{LDecomp1}
Let  $ \ell \in \mathbb{N}.$ we have the following decomposition:
 \begin{equation} \label{DecompositionL1}
\begin{split}
 L_{T}(s) & = L_{2\ell}(s) \times G_{\ell}(s), \\ 
 \end{split}
\end{equation}
where $L_{\ell}(s)$ is given in \eqref{LEven} and $G_{\ell}(s)$  is given in  terms of an Euler product which converges absolutely and uniformly for $\Re(s)>\frac{1}{2}$ and $G_{\ell}(s) \neq 0$ for $\Re(s)=1.$
\end{lem}

Before proving \lemref{LDecomp} and \lemref{LDecomp1}, we state the following result which explicitly govern  the proof of  \lemref{LDecomp}.
\begin{lem}\cite[Lemma 2.2]{LalitManish}\label{ChebpolyLem}
Let $\ell \in \mathbb{N}.$ For each $j$ with $0 \le j \le \ell$ and  $j \equiv \ell \pmod 2$, let $A_{\ell,j} := {\ell \choose {\frac{\ell-j}{2}}} - {\ell \choose {\frac{\ell-j}{2}-1}}$, and  $T_{m}(2x):= U_{m}(x)$ where $U_{m}(x)$ is the $m^{\rm th}$ Chebyshev polynomial of second kind.
Then 
\begin{equation*}
\begin{split}
x^{\ell} &= \sum_{j=0}^{\ell} A_{\ell,j} T_{\ell-j}(x). \\
\end{split}
\end{equation*}
\end{lem}

 \subsection{Proof of \lemref{LDecomp}}
 We know that $ \lambda_{f \otimes f \otimes  \cdots \otimes_{\ell} f}(n)$ is a  multiplicative function and  at prime $p$, $ \lambda_{f \otimes f \otimes  \cdots \otimes_{\ell} f}(p) = \lambda_{f}^{\ell}(p)$.   So,  $ L_{S}(s)$ is given in terms of  an Euler product, i.e.,
 \begin{equation*}
\begin{split}
L_{S}(s) =  \displaystyle{\sideset{}{^{\flat }} \sum_{n \ge 1}} ~~ \frac{\lambda_{f \otimes f \otimes  \cdots \otimes_{\ell} f}(p)}{n^{s}}
 =  \displaystyle{\prod_{p} \left(\!1+ \frac{(\lambda_{f}(p))^{\ell}}{p^s} \right)} 
 \end{split}
\end{equation*}
for $\Re(s)>1$. From Deligne's estimate, it is well-known that 
 $$\lambda_{f}(p) = 2\cos \theta \quad  {\rm  and} \quad  \lambda_{f}(p^{\ell}) = T_{\ell} (2 \cos \theta)$$ 
where  $T_{m}(2x) = U_{m}(x)$ and $U_{m}(x)$ is the $m^{\rm th}$ Chebyshev polynomial of second kind. Hence,  from \lemref{ChebpolyLem}, we get 
\begin{equation}\label{FCRel}
\begin{split}
 \lambda_{f \otimes f \otimes  \cdots \otimes_{\ell} f}(p) = {\lambda_{f}}^{\ell}(p)    = \left(\sum_{n=0}^{[\ell/2]} \left( {\ell \choose n} - {\ell \choose {n-1}} \right) \lambda_{sym_{f}^{\ell-2n}}(p) \right).\\ 
\end{split}
\end{equation}
For $\Re(s) >1$, we define the function $L_{\ell}(s)$ given by
 $$L_{\ell}(s) =  \prod_{n=0}^{[\ell/2]} \left({L(s, sym^{\ell-2n}f)}^{\left({\ell \choose n}- {\ell \choose {n-1}}\right)} \right) .$$
 For any fixed integer $\ell \geq 3$, $L_{\ell}(s)$ is a finite product of powers of symmetric power $L$-functions and hence we can write
 $$ L_{\ell}(s) := \sum_{n=1}^{\infty} \frac{A(n)}{n^s}  $$
 in $\Re(s)>1$ where $A(n)$ is a multiplicative arithmetic function. 
Now we express it in terms of  an Euler product of the form  
 \begin{equation*}
\begin{split}
 \displaystyle{\prod_{p} \left(1+ \frac{A(p)}{p^s} + \frac{A(p^{2})}{p^{2s}} + \cdots \right) }, \quad {\rm \text{and we notice that} } \quad A(p) = -{\lambda_{f}}^{\ell}(p).
 \end{split}
\end{equation*}
Moreover, for each prime $p$,  we define the sequence $B(p) = 0$, for each $r \ge 2$, \linebreak  $B(p^{r}) = A(p^{r}) + A(p^{r-1}) \lambda_{f}^{\ell} (p)$. It is easy to see that $B(n) \ll n^{\epsilon}$ for any $\epsilon$. Associated to this sequence, We define the Euler product for $U_{\ell}(s) $  given by
 \begin{equation*}
\begin{split}
U_{\ell}(s) =  \displaystyle{\prod_{p} \left(1+ \frac{B(p)}{p^s} + \frac{B(p^{2})}{p^{2s}} + \cdots \right) }
 \end{split}
\end{equation*}
with  $B(p^{2}) =  A(p^{2}) - \lambda_{f}^{2\ell}(p)$. Then, it is easy to see that  
 \begin{equation*}
\begin{split}
L_{S}(s) =  L_{\ell}(s)  U_{\ell}(s).
\end{split}
\end{equation*}
This completes the proof.

 \subsection{Proof of \lemref{LDecomp1}}
 We know that $ \lambda_{f \otimes f \otimes  \cdots \otimes_{\ell} f}(n)$ is a  multiplicative function and  at prime $p$, $ \lambda_{f \otimes f \otimes  \cdots \otimes_{\ell} f}(p) = \lambda_{f}^{\ell}(p)$.   So,  $ L_{T}(s)$ is given in terms of  an Euler product, i.e.,
 \begin{equation*}
\begin{split}
L_{T}(s) =   \displaystyle{\sideset{}{^{\flat }} \sum_{n \ge 1}} ~~ \frac{\lambda_{f \otimes f \otimes  \cdots \otimes_{\ell} f}^{2}(p)}{n^{s}}
 =  \displaystyle{\prod_{p} \left(\!1+ \frac{(\lambda_{f}(p))^{2\ell}}{p^s} \right)} 
 \end{split}
\end{equation*}
for $\Re(s)>1$. Following the arguments as in the proof of \lemref{LDecomp}, we prove the decomposition of $L_{T}(s)$.

\subsection{Convexity/Sub-convexity bound and integral estimates of the $L$-functions}

\begin{lem}\label{Riemann Zeta } 
Let $\zeta(s)= \displaystyle{\sum_{n\ge 1} \frac{1}{n^{s}}}$ be the Riemann zeta function  Then, for any $\epsilon >0$, we have 
\begin{equation}
\begin{split}
\zeta(\sigma+it) &\ll_{\epsilon} (1+|t|)^{{{\rm max} \left\{ \frac{13}{42}(1-\sigma), 0 \right\} }+\epsilon}\\
\end{split}
\end{equation}
uniformly for $\frac{1}{2} \le \sigma \le 1$ and $|t| \ge 1.$  
\end{lem}
\noindent
\textbf{proof:}
The result follows from \cite[Theorem  5]{Bourgain} and Phragmen-Lindel\"{o}f convexity principle. 
 
\begin{lem}\label{Modular L function }
For any $\epsilon >0$, the sub-convexity bound of Hecke $L$-function is given by:
\begin{equation}
\begin{split}
L(\sigma+it, f) &\ll_{f,\epsilon} (1+|t|)^{{{\rm max} \left\{\frac{2}{3}(1-\sigma), 0\right\} }+\epsilon}
\end{split}
\end{equation}
uniformly for $\frac{1}{2} \le \sigma \le 1$ and $|t| \ge 1,$ and the integral moment of Hecke $L$-function is given by:
\begin{equation}\label{Fourthmoment}
\begin{split}
\int_{0}^{T} \left| L\left(\frac{5}{8}+it , f \right)\right|^{4} dt &\ll_{f,\epsilon} T^{1+\epsilon} 
\end{split}
\end{equation} 
uniformly for $T \ge 1.$
\end{lem}
\begin{proof}
Proof of the sub-convexity bound of Hecke $L$-function follows from standard argument of Phragmen - Lindel\"{o}f convexity principle and a result of A. Good  \cite[Corollary]{AGood}. For the integral estimate, we refer to \cite[Theorem 2]{AIvic}.  
\end{proof}

 \begin{lem}\cite[Corollary 2.1]{Nunes}
 For any arbitrarily small $\epsilon >0$, we have 
\begin{equation}
\begin{split}
L(\sigma+it, sym^{2}f)  &\ll_{f, \epsilon} (1+|t|)^{{{\rm max} \left\{\frac{5}{4}(1-\sigma), 0\right\} }+\epsilon}
\end{split}
\end{equation}
uniformly for $\frac{1}{2} \le \sigma \le 1$ and $|t| \ge 1$.
 \end{lem}

 \begin{lem}\label{General L fun-Con } \cite[pp. 100]{HIwaniec}
Let $L(s,F)$ be an $L$- function of degree $m \ge 2,$ i.e.,
\begin{equation}
\begin{split}
L(s, F) & = \sum_{n \ge 1} \frac{\lambda_{F}(n)}{n^{s}} = \prod_{p-{\rm prime}} \prod_{j= 1}^{m} \left(1-\frac{\alpha_{p, f, j}}{p^s}\right)^{-1},
\end{split}
\end{equation}
where $\alpha_{p, f, j}$, $1 \le j \le m$; are the local parameter of $L( s, F)$ at a prime $p$ and $\lambda_{F}(n) = O(n^{\epsilon})$ for any $\epsilon>0.$ The series and Euler product for $L(s, F)$ converge  absolutely for $\Re(s) > 1$. Suppose  $L(s, F)$ is an entire function except possibly for pole at $s = 1$ of order $r$ and satisfies a nice functional equation $(s \rightarrow 1-s)$. Then for any $\epsilon >0$, we have 
\begin{equation}
\begin{split}
\left( \frac{s-1}{s+1}\right)^{r}L(\sigma+it, F) &\ll_{\epsilon} (\mathcal{Q}_F(1+|t|)^{m})^{\frac{1}{2}(1-\sigma)+\epsilon}
\end{split}
\end{equation}
uniformly for $0 \le \sigma \le 1$,  and $|t| \ge 1$ where $s =\sigma+it$. For $T\ge 1$, We have 
\begin{equation}
\begin{split}
\int_{T}^{2T} \left|L \left(\sigma+it, F\right)\right|^{2} dt &\ll_{\epsilon} (\mathcal{Q}_F(1+|t|)^{m})^{(1-\sigma)+\epsilon}
\end{split}
\end{equation}
uniformly for $\frac{1}{2} \le \sigma \le 1$ where $\mathcal{Q}_F$ is the analytic conductor of $F$. 
\end{lem}

\section{Proof of Results}

\subsection{General Philosophy}
  Let $1 \le Y < \frac{X}{2}.$ We introduce a smooth compactly supported function $w(x)$ satisfying: $w(x) =1$ for $x \in [2Y, X],$ $w(x) = 0$ for $x<Y$ and $x> X+Y,$ and $w^{(r)}(x) \ll_{r} Y^{-r}$ for all $r\ge 0.$ Following the idea of \cite{YJGL},   we  sketch the method in \cite{Lalit-M} to get results associated to any arithmetical function $n \mapsto f(n)$. Moreover let $f(n) \ll n^{\epsilon}$ for any arbitrarily small $\epsilon>0$. Then, from \cite[Section 4.1]{Lalit-M}, we get 
 \begin{equation}\label{Mainest}
\sum_{n \le X} f(n) = \underset{s=1}{\rm Res} \left( \frac{X^{s}}{s}\sum_{n\ge 1}\frac{f(n)}{n^{s}} \right) + |V| + O(X^{-A}) + O(Y^{1+\epsilon})
\end{equation}
with (for any fixed $\sigma_{0} \in (1/2, 1)$) 
\begin{equation}\label{CRT1/2}
V=\frac{1}{2 \pi i} \int_{\sigma_{0}-iT}^{\sigma_{0}+iT} \tilde w(s) \left(\sum_{n\ge 1} \frac{f(n)}{n^{s}}\right)  ds,
\end{equation} 
where $\tilde w(s)$ is the the Mellin's transform  of $w(t)$ and   $T= \frac{X^{1+\epsilon}}{Y}$ and $Y$ to be chosen later. $\tilde w(s)$, the Mellin's transform  of $w(t)$, is given by the integral:
$
\tilde w(s) = \int_{0}^{\infty} w(x) x^{s} \frac{dx}{x},
$ 
satisfying 
\begin{equation}\label{FourierW}
\tilde w(s) =  \frac{1}{s(s+1)\cdots(s+m-1)}\int_{0}^{\infty} w^{(m)}(x) x^{s+m-1} dx \ll \frac{Y}{X^{1-\sigma}}  \left(\frac{X}{|s|Y}\right)^{m}
\end{equation}
for any $m\ge 1,$ where $\sigma = \Re(s)$. Hence, it is enough to find an upper  estimate for $V$ and residue of associated $L$-function at $s=1$ to get the required result.

\noindent
\subsection{Proof of \thmref{Approx1}:}

From the Deligne's bound,  we known that $\lambda_{f\otimes f \otimes \cdots \otimes_{\ell} f}(n)  \ll n^{\epsilon}$ for any $\epsilon>0$.  Then,
from equation \eqref{Mainest} with  $f(n) = \lambda_{f\otimes f \otimes \cdots \otimes_{\ell} f}(n) $, we have 
 \begin{equation}\label{EstFN}
 \begin{split}
S_{\ell}(f, X ) & = \sideset{}{^{\flat }} \sum_{n \le X}\lambda_{f\otimes f \otimes \cdots \otimes_{\ell} f}(n) = \underset{s=1}{\rm Res} \left( \frac{X^{s}}{s} L_{S}(s) \right) + |V_{\ell}| + O(X^{-A}) + O(Y^{1+\epsilon}) \\
{\rm with} \quad 
V_{\ell} &=\frac{1}{2 \pi i} \int_{\sigma_{0}-iT}^{\sigma_{0}+iT} \tilde w(s)  L_{S}(s)  ds \\
\end{split}
\end{equation}
for any fixed $\sigma_{0} \in (1/2, 1)$. We  substitute the decomposition of $L_{S}(s)$ ($ L_{S}(s) = L_{\ell}(s) \times U_{\ell}(s)$ where $L_{\ell}(s)$ and $ U_{\ell}(s)$ are given in  \lemref{LDecomp}) in  \eqref{EstFN} ,  and use the absolute convergence of $U_{\ell}(s)$ in $\Re(s)> \sigma_{0}$ and bound for $\tilde w(s)$ from \eqref{FourierW}, to get
\begin{equation*}
\begin{split}
|V_{\ell}| &\ll X^{\sigma_{0}} \int_{-T}^{T}  \frac{|L_{\ell} \left( \sigma_{0} +it \right)|}{|\sigma_{0} +it|}  dt \ll 2X^{\sigma_{0}} \int_{0}^{T}  \frac{|L_{\ell} \left( \sigma_{0} +it \right)|}{| \frac{1}{2}+ \epsilon +it|}  dt \\
& \ll X^{\sigma_{0}} \left\{ \int_{0}^{1} +  \int_{1}^{T} \right\} \frac{|L_{\ell} \left( \sigma_{0} +it \right)|}{|\sigma_{0} +it|}  dt.
\end{split}
\end{equation*}
In first integral, we substitute the respective bound and in second integral, we appeal dyadic division method to get 
\begin{equation}\label{VLEst}
\begin{split}
|V_{\ell}| 
& \ll  X^{\sigma_{0}} + X^{\sigma_{0}} \log T \underset{ 2  \le T_{1} \le  T }{\rm \max} (I_{\ell}(T_{1})) 
\end{split}
\end{equation}
\begin{equation}\label{VLEstInt}
\begin{split}
\quad {\rm where } \quad
 I_{\ell}(T) =   \frac{1}{T}  \int_{T/2}^{T}  |L_{\ell} \left( \sigma_{0} +it \right) | dt.  
\end{split}
\end{equation}
Thus,  the estimate for $I_{\ell}(T)$ leads to required estimate for $S_{\ell}(f, X)$.

\smallskip
\noindent
\textbf{Case 1 ($\ell$ is odd):}
\smallskip
We take $\sigma_{0} = 5/8$ and substitute the decomposition of $L_{\ell}(s)$ when $\ell$ is odd from \eqref{LOdd} in  \eqref{VLEstInt} and apply Cauchy-Schwarz inequality  to get 
\begin{equation*}
 \begin{split}
& |I_{\ell}(T)|= \frac{1}{T}  \int_{T/2}^{T}  L_{\ell} \left( \sigma_{0} +it \right) dt.  \ll
\begin{cases}
 \frac{1}{T}  \left( \int_{\frac{T}{2}}^{T}   \left| {L(\sigma_{0} +it, f)}  \right|^{2\left({\ell \choose [\ell/2]}- {\ell \choose {[\ell/2]-1}}\right)}  dt  \right)^{\frac{1}{2}} \\
 \left( \int_{\frac{T}{2}}^{T}  \left| \displaystyle{\prod_{n=0}^{[\ell/2]-1}} L(\sigma_{0} +it, sym^{\ell-2n}f )^{\left({\ell \choose n}- {\ell \choose {n-1}}\right)} \right|^{2}   dt  \right)^{\frac{1}{2}}. \\
  \end{cases} \\
& |I_{\ell}(T)|  \ll
\begin{cases}
   \frac{1}{T}   \underset{ \frac{T}{2}  \le t \le  T }{\rm sup}  \left( |L(\sigma_{0} +it, f)|^{\left({\ell \choose [\ell/2]}- {\ell \choose {[\ell/2]-1}}-2\right)}  \right)  
   \times \left( \int_{\frac{T}{2}}^{T}   \left| {L(\sigma_{0} +it, f)}  \right|^{4} dt  \right)^{\frac{1}{2}} \\
 \times   \left( \int_{\frac{T}{2}}^{T}  \left| \displaystyle{\prod_{n=0}^{[\ell/2]-1}} L(\sigma_{0} +it, sym^{\ell-2n}f)^{\left({\ell \choose n}- {\ell \choose {n-1}}\right)} \right|^{2}   dt  \right)^{\frac{1}{2}}. \\
  \end{cases}
 \end{split}
\end{equation*}
 We apply the convexity bound/sub-convexity bound and fourth integral moment of Hecke $L$-function and above identity  to get  
\begin{equation*}
 \begin{split}
 |I_{\ell}(T)| 
 &  \ll 
   T^{-1}T ^{\left({\ell \choose [\ell/2]}- {\ell \choose {[\ell/2]-1}}-2\right)\frac{2}{3}(1-\frac{5}{8}) } T^{\frac{1}{2}+\epsilon}  T^{ \left(\frac{3}{2. 8}  \displaystyle{\sum_{n=0}^{[\ell/2]-1}} (\ell-2n+1) {\left({\ell \choose n}- {\ell \choose {n-1}}\right)} \right)} \\
 \end{split}
\end{equation*}
\begin{equation*}
 \begin{split}
 |I_{\ell}(T)| 
 &  \ll 
   T^{-1 + \frac{1}{4} \left({\ell \choose [\ell/2]}- {\ell \choose {[\ell/2]-1}}-2\right) + \frac{1}{2}+\epsilon + \frac{3}{16} \left(  \displaystyle{\sum_{n=0}^{[\ell/2]-1}} (\ell-2n+1) {\left({\ell \choose n}- {\ell \choose {n-1}}\right)} \right)} \\
 &  \ll 
  {T}^{\left( \frac{1}{4([\frac{\ell}{2}]+2)} {\ell \choose {[\frac{\ell}{2}]}} + \frac{3}{16} \displaystyle{ \left[\sum_{n=0}^{[\ell/2]-1} {\frac{(\ell-2n+1)^{2}}{\ell-n+1} {\ell \choose {n}}} \right]} -1 \right)},
 \end{split}
\end{equation*}
which follows from  the identity $ {\ell \choose {n}}- {\ell \choose {n-1}} = \frac{\ell-2n+1}{\ell-n+1} {\ell \choose {n}}$ when $n >0$ and $1$ when $n=0$. The proof of above identity follows clearly from the definition. We substitute the value of  $|I_{\ell}(T)|$ to get
\begin{equation}\label{VEst}
 \begin{split}
 |V_{l}|  
 &  \ll X^{\frac{5}{8}} +  X^{\frac{5}{8}}
   {T}^{\left( \frac{1}{4([\frac{\ell}{2}]+2)} {\ell \choose {[\frac{\ell}{2}]}} + \frac{3}{16} \displaystyle{ \left[\sum_{n=0}^{[\ell/2]-1} {\frac{(\ell-2n+1)^{2}}{\ell-n+1} {\ell \choose {n}}} \right]} -1 \right)}.
 \end{split}
\end{equation}
We know that the function $L_{S}(s)$ is holomorphic $\Re(s)>1/2$. Thus, substituting  the estimate of $V_{\ell}$ from \eqref{VEst} in \eqref{EstFN},  we have (for odd $\ell$)
 \begin{equation*}
S_{\ell}(f, X )= \sideset{}{^{\flat }} \sum_{n \le X}\lambda_{f\otimes f \otimes \cdots \otimes_{\ell} f}(n) = O\left( X^{\frac{5}{8}} T^{\tilde{\alpha_{\ell}}-1}  \right)   + O(Y^{1+\epsilon}) + O(X^{-A})
\end{equation*}
where $\tilde\alpha_{\ell} = \frac{3}{8}{\left( \frac{2}{3([\frac{\ell}{2}]+2)} {\ell \choose {[\frac{\ell}{2}]}} + \frac{1}{2} \displaystyle{ \left[\sum_{n=0}^{[\ell/2]-1} {\frac{(\ell-2n+1)^{2}}{\ell-n+1} {\ell \choose {n}}} \right]} \right)}.$ We substitute $T= \frac{X^{1+\epsilon}}{Y}$ and  choose $Y = X^{1- \frac{3}{8 \tilde\alpha_{\ell}}+ \epsilon} $ to get 
 \begin{equation*}
S_{\ell}(f, X )=  O\left( X^{1- \frac{1}{\alpha_{\ell}}+ \epsilon} \right)
\end{equation*}
where $\alpha_{\ell} = \frac{8}{3} \tilde \alpha_{\ell}$. This completes the proof when $\ell$ is odd. 

\smallskip
\noindent
{\bf Case 2 ($\ell$ is even):}
We take $\sigma_{0} = \frac{1}{2}+\epsilon$ and substitute the decomposition of $L_{\ell}(s)$ (from \eqref{LEven}) when $\ell $ is even in  \eqref{VLEst} and apply Cauchy-Schwarz inequality  to get 
\begin{equation*}
 \begin{split}
 |I_{\ell}(T)| 
  & \ll \frac{1}{T} \int_{\frac{T}{2}}^{T} |L_{\ell}(\sigma_{0} + i t)| dt \ll \frac{1}{T}  \int_{\frac{T}{2}}^{T}  \left| \prod_{n=0}^{\ell/2}  \left({L(\sigma_{0} +it, sym^{\ell-2n}f)} \right)^{\left({\ell \choose n}- {\ell \choose {n-1}}\right)} \right| dt  \\
& \ll
\begin{cases}
\frac{1}{T}   \underset{ \frac{T}{2}  \le t \le  T}{\rm sup}  \left( |\zeta(\sigma_{0} +it)|^{\left({\ell \choose \ell/2}- {\ell \choose {\ell/2-1}}\right)}  \left| {L(\sigma_{0} +it, sym^{2}f)} \right|^{\left({\ell \choose \ell/2-1}- {\ell \choose {\ell/2-2}}-1\right)} \right) \\ 
 \left(  \int_{\frac{T}{2}}^{T}  \left|{L(\sigma_{0} +it, sym^{2}f)} \right|^{2}  dt  \right)^{\frac{1}{2}} 
  \left(  \int_{\frac{T}{2}}^{T}  \left| \displaystyle{ \prod_{n=0}^{\ell/2-2}} {L(\sigma_{0} +it, sym^{\ell-2n}f)}^{\left({\ell \choose n}- {\ell \choose {n-1}}\right)} \right|^{2}  dt  \right)^{\frac{1}{2}}. \\
  \end{cases}
 \end{split}
\end{equation*}
Using the appropriate sub-convexity bound for $\zeta(s)$ and $L(s, sym^{2}f)$ and integral estimate for general $L$-functions, we have
\begin{equation*}
 \begin{split}
 |I_{\ell}(T)| 
 &  \ll 
  {T}^{\left( \frac{13}{42 (\ell+2)} {\ell \choose {\frac{\ell}{2}}} +  \frac{15}{4(\ell+4)} {\ell \choose {\frac{\ell}{2}}-1} + \frac{1}{4} \displaystyle{ \left[\sum_{n=0}^{\frac{\ell}{2}-2} {\frac{(\ell-2n+1)^{2}}{\ell-n+1} {\ell \choose {n}}} \right]} -\frac{7}{8} \right)}.
 \end{split}
\end{equation*}
We substitute the value of  $|I_{\ell}(T)|$ to get
\begin{equation}\label{VEstEven}
 \begin{split}
 |V_{\ell}|  
 &  \ll X^{\frac{1}{2}+\epsilon} + X^{\frac{1}{2}+\epsilon}
    {T}^{\left( \frac{13}{42 (\ell+2)} {\ell \choose {\frac{\ell}{2}}} +  \frac{15}{4(\ell+4)} {\ell \choose {\frac{\ell}{2}}-1} + \frac{1}{4} \displaystyle{ \left[\sum_{n=0}^{\frac{\ell}{2}-2} {\frac{(\ell-2n+1)^{2}}{\ell-n+1} {\ell \choose {n}}} \right]} -\frac{7}{8} \right)}.
 \end{split}
\end{equation}
We know that the function $L_{S}(s)$ has a pole at $s=1$, of order  $ \frac{2}{ (\ell+2)} {\ell \choose {\frac{\ell}{2}}}-1$. Thus, substituting  the estimate of $V_{\ell}$ from \eqref{VEstEven} in \eqref{EstFN},  we have (for even $\ell$)
 \begin{equation*}
 \begin{split}
S_{\ell}(f, X ) & =  X P_{\ell}(\log X) + O\left( X^{\frac{1}{2}+\epsilon}  \right) + O\left( X^{\frac{1}{2}+\epsilon}
   {T}^{\left(\gamma_{\ell} -1 \right)} \right) +  O\left( Y^{1+\epsilon}\right)+ 
 O(X^{-A})   
  \end{split}
\end{equation*}
where $\gamma_{\ell} =  \frac{1}{8}+\frac{13}{42 (\ell+2)} {\ell \choose {\frac{\ell}{2}}} +  \frac{15}{4(\ell+4)} {\ell \choose {\frac{\ell}{2}}-1} + \frac{1}{4} \displaystyle{ \left[\sum_{n=0}^{\frac{\ell}{2}-2} {\frac{(\ell-2n+1)^{2}}{\ell-n+1} {\ell \choose {n}}} \right]}.$ We substitute  $T= \frac{X^{1+\epsilon}}{Y}$ and  choose $Y = X^{1- \frac{1}{2\gamma_{\ell}}+ \epsilon} $ to get 
 \begin{equation*}
 \begin{split}
S_{\ell}(f,  X ) & =  X P_{\ell}(\log X) + O\left( X^{1- \frac{1}{2 \gamma_{\ell}}+ \epsilon} \right) =  X P_{\ell}(\log X) + O\left( X^{1- \frac{1}{\beta_{\ell}}+ \epsilon} \right)
  \end{split}
\end{equation*}
where $\beta_{\ell} = 2\gamma_{\ell}=  \frac{1}{4}+\frac{13}{21 (\ell+2)} {\ell \choose {\frac{\ell}{2}}} +  \frac{15}{2(\ell+4)} {\ell \choose {\frac{\ell}{2}}-1} + \frac{1}{2} \displaystyle{ \left[\sum_{n=0}^{\frac{\ell}{2}-2} {\frac{(\ell-2n+1)^{2}}{\ell-n+1} {\ell \choose {n}}} \right]}.$
This completes the proof when $\ell$ is even. 

\subsection{Proof of \thmref{Approx2}:} From the \lemref{LDecomp1}, we know that the Dirichlet series $L_{T}(s)$ associated to sum $T_{\ell}(f, X )$ is given by
$ L_{T}(s)  = L_{2\ell}(s) \times G_{\ell}(s) $
where $L_{\ell}(s)$ is given in \eqref{LEven} and $G_{\ell}(s)$  is given in  terms of some Euler product which converges absolutely and uniformly for $\Re(s)>\frac{1}{2}$ and $G_{\ell}(s) \neq 0$ for $\Re(s)=1.$ Hence, following the argument as in the proof of  \thmref{Approx2} (to establish the estimate for $S_{\ell}(f, X )$ when $\ell$ is even), we prove our estimate.

\subsection{Proof of \thmref{N sign change}:}
Let us consider  $h= X^{1- \delta}$ with $  \frac{1}{\beta_{2\ell}} \le \delta < \frac{1}{\alpha_{\ell}}$. Assume that   the sequences $\{\lambda_{f \otimes f \otimes \cdots \otimes_{\ell} f}(n)\}_{n- {\rm squarefree}}$  has a  constant sign (say positive) supported in the interval $ (X,X+X^{1-\delta}]$. Then, using Deligne's bound, i.e.,  $\lambda_{f \otimes f \otimes \cdots \otimes_{\ell} f}(n)  \ll n^{\epsilon}$ for any arbitrary small $\epsilon>0$, we have  
\begin{equation}\label{FCsign}
\begin{split}
 T_{\ell}(f, X+h)- T_{\ell}(f,  X) & = \sideset{}{^{\flat }} \sum_{X \le n \le X+h} \lambda_{f \otimes f \otimes \cdots \otimes_{\ell} f}^{2}(n)  \ll X^{\epsilon}  \sideset{}{^{\flat }} \sum_{X \le n \le X+h}  \lambda_{f \otimes f \otimes \cdots \otimes_{\ell} f}(n) \\
&    = X^{\epsilon} ( S_{\ell}(f,  X+h)- S_{\ell}(f,  X)). \\
&   \ll X^{\epsilon} \left( (X+h)^{1- \frac{1}{\alpha_{\ell}}+ \epsilon} + X^{1- \frac{1}{\alpha_{\ell}}+ \epsilon} \right) \ll  X^{1- \frac{1}{\alpha_{\ell}}+ \epsilon},
\end{split}
\end{equation}
which follows from \eqref{EstT}. Moreover,  from equation \eqref{EstS}, we have
\begin{align}\label{FCCsign}
  T_{\ell}(f, X+h)- T_{\ell}(f,  X) &  = (X+h) P_{\ell}(\log(X+h))  - X P_{\ell}(\log X) + O\left( X^{1- \frac{1}{2\beta_{\ell}}+ \epsilon} \right)\nonumber  \\
& \ge (X+h) P_{\ell}(\log X)   - X P_{\ell}(\log X) + O\left( X^{1- \frac{1}{\beta_{2\ell}}+ \epsilon} \right) \\
&= h P_{\ell}(\log X)+ O\left( X^{1- \frac{1}{\beta_{2\ell}}+ \epsilon} \right) \gg   X^{1-\delta}. \nonumber 
\end{align}
Thus, comparing the estimates given in the  \eqref{FCsign} and \eqref{FCCsign},  we   arrive at a contradiction. Thus, we have a sign change of the sequence $\{\lambda_{f \otimes f \otimes \cdots \otimes_{\ell} f}(n)\}_{n- {\rm squarefree}}$ in the interval  $(X,X+X^{1-\delta}]$.  Moreover we have (at least)  $X^{\delta}$  sign changes in the interval $[X, 2X]$.
 
\bigskip
\noindent
\textbf{Acknowledgement :} The  authors would like to thank IMSc, Chennai for its warm hospitality and wonderful academic atmosphere. The author (1) is thankful to IMSc, Chennai for its generous support during his visit in summer 2023.

\end{document}